\documentclass[10pt, reqno]{amsart}

\usepackage{amsmath,amssymb,amsthm,amsfonts,verbatim}
\usepackage{graphicx}
\usepackage{pinlabel}
\usepackage{hyperref}
\usepackage{color}
\usepackage{enumerate}
\usepackage{cite}
\usepackage{subcaption}
\usepackage{tikz}
\usepackage{hyperref}
\hypersetup{colorlinks=true, citecolor=blue, linkcolor=blue, urlcolor=blue}

\usepackage[top=1.2in,bottom=1.8in,left=1.2in,right=1.2in]{geometry}

\theoremstyle{definition}
\newtheorem{theorem}{Theorem}[section]
\newtheorem{definition}[theorem]{Definition}

\newtheorem{conjecture}[theorem]{Conjecture}
\newtheorem{lemma}[theorem]{Lemma}

\newtheorem{question}[theorem]{Question}
\newtheorem{proposition}[theorem]{Proposition}

\title{Geodesic nets on flat spheres}
\author{Ian Adelstein\textsuperscript{*}, Elijah Fromm, Rajiv Nelakanti, Faren Roth, Supriya Weiss}

\address{Department of Mathematics, Yale University, 10 Hillhouse Ave, New Haven, CT 06511}

\thanks{\noindent\textsuperscript{*}corresponding author ian.adelstein@yale.edu}

\begin{document}

\begin{abstract} We consider geodesic nets (critical points of a length functional on the space of embedded graphs) on doubled polygons (topological $2$-spheres endowed with a flat metric away from finitely many cone singularities). We use the theorem of Gauss-Bonnet to demonstrate the existence and non-existence of specific geodesic nets on regular doubled polygons.
\end{abstract}

\maketitle

\section{Introduction}


The existence of closed geodesics on Riemannian manifolds is a foundational question in the field of differential geometry dating back to early work of Poincar\'e. A fruitful approach has been to realize closed geodesics as critical points of a length functional on the space of loops on a Riemannian manifold. If one instead considers a length functional on the space of embedded graphs, then the critical points form a new class of stationary objects:~geodesic nets.

\begin{definition}
A geodesic net is a graph embedded into a manifold such that each edge is a geodesic and at each vertex the unit tangent vectors to the edges sum to zero. 
\end{definition}

We note that geodesic nets are critical points (not necessarily minima) of a length functional, just like geodesics. The edge condition (that they be geodesics) ensures that geodesic nets are locally shortest paths, and the vertex condition ensures that the edges are balanced; that is, there is no direction in which to perturb the net and decrease its total length. For example, at a vertex of degree 3, there is a $\frac{2\pi}{3}$ angle between any two neighboring edges

\begin{figure}[hb]
\centering
\includegraphics[scale=.4]{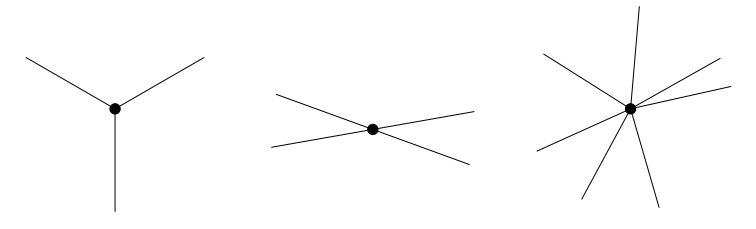}
\caption{Balanced vertices of degree 3, 4, and 7; image credit \cite{NP}.}
\label{fig:balanced}
\end{figure}

As a first (almost trivial) example, a closed geodesic is a geodesic net. 
Any vertex introduced along a closed geodesic 
will satisfy the balancing condition, as the two unit tangent vectors point in opposite directions, hence sum to zero. For this reason we do not consider vertices of degree two when discussing geodesic nets.

A first non-trivial example is a theta-graph, which is a geodesic net modeled on a graph with exactly two vertices connected by exactly three edges. Figure~\ref{fig:theta} illustrates such a theta-graph on a round sphere. Some of the only known existence results for geodesic nets (outside of specific examples) come from Hass and Morgan. Their main result in \cite{HM} demonstrates that every Riemannian 2-sphere with positive curvature admits a geodesic net modeled on one of the graphs from Figure~\ref{fig:types}, and as a corollary they conclude that convex metrics on the 2-sphere sufficiently close to the round metric necessarily admit a theta-graph.

\begin{figure}[ht!]
\centering
\includegraphics[scale=.38]{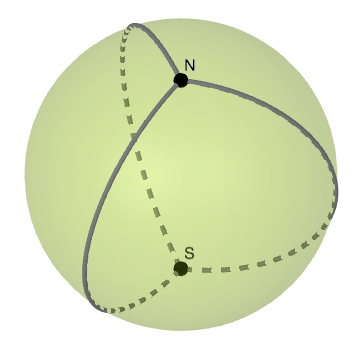}
\caption{A theta-graph on a round sphere; image credit \cite{Ade}.}
\label{fig:theta}
\end{figure}

\begin{figure}[ht!]
\centering
\includegraphics[scale=.54]{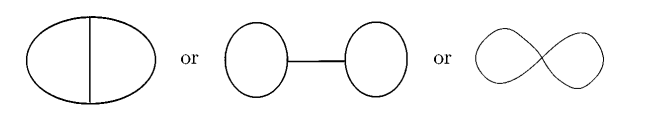}
\caption{A theta-graph, a bifocal, and a figure-eight; image credit \cite{HM}.}
\label{fig:types}
\end{figure}



The graphs in Figure~\ref{fig:types} are of particular interest as they are the only graphs which partition a 2-sphere into three regions (see Lemma~\ref{lem}). We consider geodesic nets modeled on these graphs and demonstrate the existence and nonexistence of such nets on a specific class of 2-spheres:~the doubled regular polygons. A doubled polygon is the metric space obtained by gluing two identical n-gons along their common boundary. As such, a doubled polygon is a topological 2-sphere with a flat metric away from finitely many cone singularities.

We first consider the space of all doubled triangles and find that exactly one triangle admits a theta-graph and exactly one triangle admits a bifocal graph, whereas all isosceles triangles admit figure-eight graphs. We then show that the only regular $n$-gons to admit 3-regular geodesic nets are 3$n$- and 4$n$-gons. We find that only regular 3$n$-gons may admit theta-graphs (and they do), only regular 12$n$-gons may admit bifocal graphs, and all regular odd-gons admit figure-eight graphs. Finally, we study which regular even-gons do and do not admit figure-eight graphs. 

\section{Preliminary Arguments}


Our main tool to investigate the existence of geodesic nets on doubled polygons is the theorem of Gauss-Bonnet, which relates the geometry and topology of a surface. 

\begin{theorem}
Let $M$ be a surface with boundary $\partial M$. Let $K$ be Gaussian curvature, $k_g$ geodesic curvature, and $\chi (M)$ the Euler characteristic. Then
$$\iint_{M} K ~dA + \int_{\partial M} k_g ~ds = 2\pi \chi (M).$$
\end{theorem}

We formulate an analogue of this theorem relevant to the study of geodesic nets on doubled polygons. 
Topologically, the doubled $n$-gon is a sphere with total Gaussian curvature $4\pi$. On a round sphere, this curvature is uniformly distributed on the surface. On a doubled regular $n$-gon, this curvature is concentrated in the vertices, with each vertex receiving $ 4\pi /n$ curvature. 
A (connected) geodesic net on a doubled polygon divides the space into a finite number of faces, $F$, each of which is topologically a disk and therefore has $\chi(F)=1$. The edges of a face are geodesic, so that $$\int_{\partial F} k_g ~ds = \sum_i \alpha_i$$ where $\alpha_i$ are the exterior turning angles at the vertices of the face. Since curvature is equally concentrated at each vertex, a face that encloses $x$ vertices of the doubled $n-$gon will have $\frac{4\pi}{n}x$ curvature. For any of these faces $F$ the Gauss-Bonnet formula thus simplifies to 
\begin{equation}\label{gb}
    \frac{4\pi}{n}x + \sum_i \alpha_i = 2\pi .
\end{equation}

In the case of a 3-regular geodesic net we have that the exterior turning angle $\alpha_i = \pi/3$ at every vertex. For a face $F$ that has $y$ edges (and thus $y$ vertices) and which encloses $x$ vertices of the doubled $n$-gon the Gauss-Bonnet formula further simplifies to
\begin{align*}
\frac{4\pi}{n}x+\frac{\pi}{3}y &= 2\pi.
\end{align*}
Solving this expression for $n$ yields
\begin{align}\label{3reg}
n=\frac{12x}{6-y}
\end{align}
and we see that each face of a 3-regular geodesic net restricts the value of $n$ in terms of its $y$ edges and the $x$ vertices it encloses. Note in the degenerate case when $y=6,$ $\frac{4\pi}{n}x = 0,$ so this face encloses no vertices of the polygon and imposes no restriction on $n.$

We use Equation~\ref{3reg} and the fact that $x, y, n$ are all integer-valued to explore the existence and non-existence of 3-regular geodesics nets on doubled regular $n$-gons in the next section.

We conclude the preliminary arguments with the following:

\begin{lemma}\label{lem}
The theta-graph, figure-eight, and bifocal (see Figure~\ref{fig:types}) are the only graphs which partition a 2-sphere into three regions
\end{lemma}
\begin{proof}
This fact follows from the definition of the Euler characteristic, $V-E+F=\chi$. For a partition with 3 faces on a 2-sphere (where $\chi=2$) this formula reduces to $E=V+1$. As we (definitionally) require each vertex to have degree at least 3, we recover the inequality $3V \leq 2E$ and thus $V \leq 2$.
When $V=1$, there must be two loops (figure-eight); when $V=2$, there must be three edges (theta-graph) or two loops and one edge (bifocal). 
\end{proof}

\section{Geodesic Nets on Doubled Triangles}

We begin by classifying the existence of theta-graphs, figure-eight graphs, and bifocal graphs on irregular doubled triangles. The combinatorics of the graph, together with Gauss-Bonnet (Equation~\ref{gb}), determine the curvature in each of the three faces. In particular, each face has positive curvature and thus must contain exactly one vertex of the triangle. This condition determines the angles that are possible at each vertex. 

\begin{theorem}
There exists an embedded theta-graph on a doubled triangle if and only if the triangle is equilateral.
\end{theorem}

\begin{proof}
Each face of the theta-graph contains a pair of vertices with angle $\frac{2\pi}{3}$ each. Gauss-Bonnet (Equation~\ref{gb}) then implies that each face of the partition contains $2\pi - 2(\frac{\pi}{3}) =  \frac{4\pi}{3}$ curvature. This is only happens when the triangle is equilateral. For existence see Figure~\ref{existence}. 
\end{proof}

\begin{theorem}
There exists a bifocal graph on a doubled triangle if and only if the triangle is a 30-30-120 triangle.
\end{theorem}

\begin{proof}
By Gauss-Bonnet (Equation~\ref{gb}) the two loops of the bifocal contain $\frac{5\pi}{3}$ curvature each, and the exterior face contains $\frac{2\pi}{3}$ curvature. This only happens for the 30-30-120 triangle. For existence see Figure~\ref{isoc}. 
\end{proof}


\begin{figure}
    \centering
    \begin{minipage}{0.56\textwidth}
        \centering
        \includegraphics[width=1\textwidth]{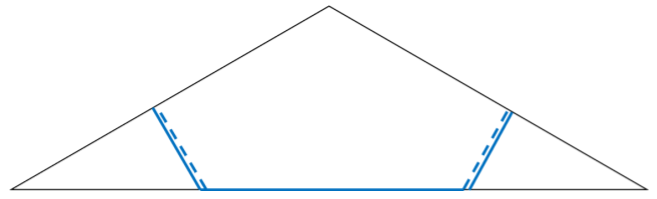} 
    \end{minipage}\hfill
    \begin{minipage}{0.44\textwidth}
        \centering
        \includegraphics[width=0.76\textwidth]{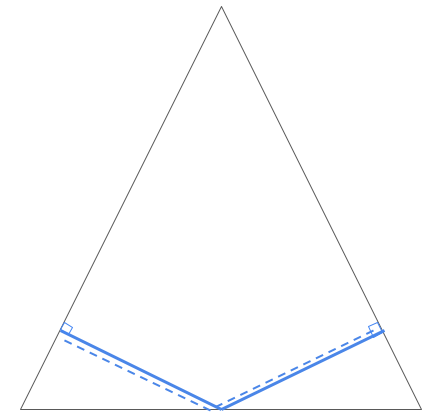} 
    \end{minipage}
    \caption{Bifocal on 30-30-120 and figure-eight on isosceles.}
    \label{isoc}
\end{figure}

\begin{theorem}
There exists a figure-eight graph on a doubled triangle if and only if the triangle is isosceles.
\end{theorem}

\begin{proof}
To satisfy the balancing condition of geodesic nets, the two loops of a figure-eight have the same exterior turning angle, hence by Gauss-Bonnet must contain the same amount of curvature. This is only possible if the triangle is isosceles. For existence, one can construct a figure-eight on an isosceles triangle by placing the single point of self-intersection at the midpoint of the edge which sits between the two equal angled vertices. Then shoot the geodesic from this midpoint such that it intersects its next edge perpendicularly, and continue smoothly to mirror this behavior perpendicularly through the final edge (see Figure~\ref{isoc}). 
\end{proof}


\section{3-Regular Graphs on Doubled Regular Polygons}
Next, we consider doubled $n-$gons for $n\geq 3$, restricting our attention to doubled regular $n-$gons. We begin with a general result that restricts 3-regular geodesic nets to the doubled regular $3n$ and $4n$-gons. 

\begin{theorem}
There exists a 3-regular geodesic net on a doubled regular polygon if and only if the polygon is a $3n$ or $4n$-gon.
\end{theorem}

\begin{proof}
As all faces must have $y>0$ we see from Equation~\ref{3reg} that $n$ is always a multiple of $3$ or $4$. Though the degenerate case $y=6$ recovers no restriction on $n$, any geodesic net must produce at least one face with $x>0$, and this face will yield the restriction on $n$ as above.

For the existence of 3-regular geodesic nets on $3n$ and $4n$-gons we proceed via construction. We demonstrate a 3-regular geodesic net on the doubled triangle and square in Figure~\ref{existence}; by symmetry, it is clear that these are geodesic nets. As the edges of the nets traverse the midpoints of the polygons, we can ``cut off'' the corners of the polygons without disrupting the geodesic net; we conclude that such a net exists on any $3n$ or $4n$-gon, respectively. 
\end{proof}

\begin{figure}
    \centering
    \begin{minipage}{0.5\textwidth}
        \centering
        \includegraphics[width=0.8\textwidth]{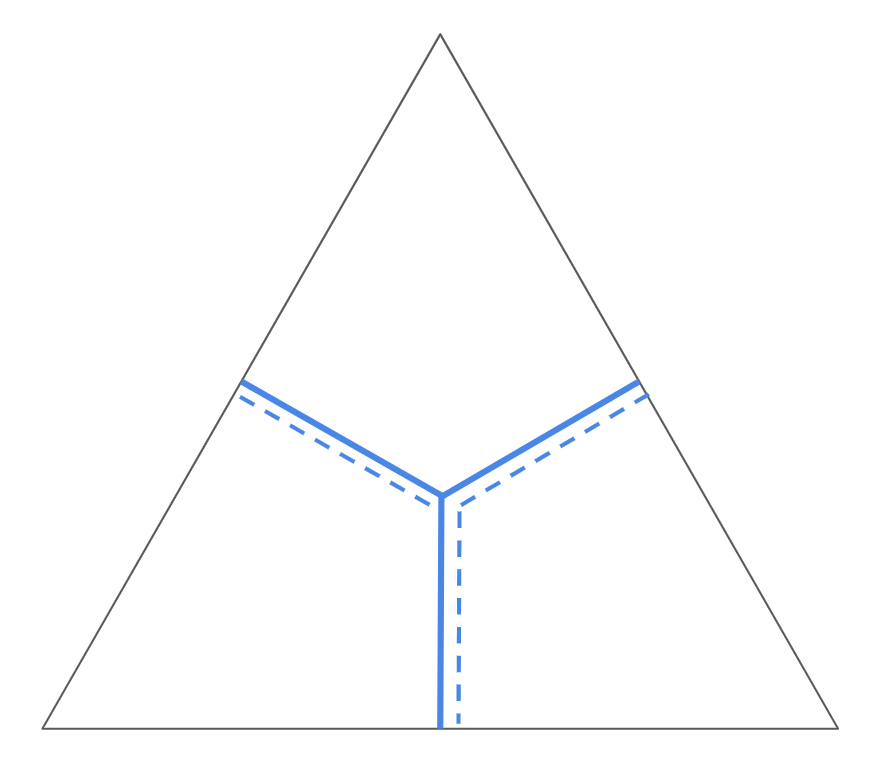} 
    \end{minipage}\hfill
    \begin{minipage}{0.5\textwidth}
        \centering
        \includegraphics[width=0.66\textwidth]{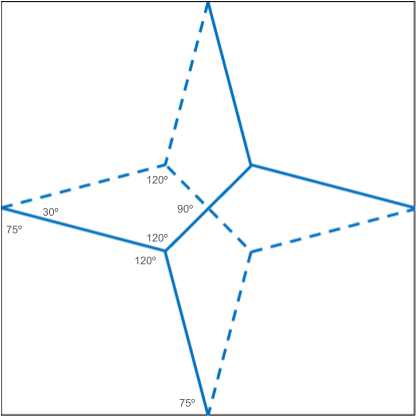} 
    \end{minipage}
    \caption{3-regular geodesic nets on the doubled triangle and square.}
    \label{existence}
\end{figure}



Next we consider the existence of some specific 3-regular graphs:~the theta-graph and the bifocal graph.

\begin{theorem}
There exists an embedded theta-graph on a doubled regular polygon if and only if the polygon is a $3n$-gon.
\end{theorem}

\begin{proof}
As all faces in a theta-graph have $y=2$ we see from Equation~\ref{3reg} that $n=3x$. The existence result follows as before from Figure~\ref{existence}.
\end{proof}





\begin{theorem}
If there exists an embedded bifocal graph on a doubled regular polygon, then the polygon is a $12n$-gon.
\end{theorem}

\begin{proof}
Consider the face defined by one of the loops of the bifocal. This face has $y=1$ and Equation~\ref{3reg} thus yields $n=12x/5$. 
\end{proof}

Note that we do not have an existence result for the bifocal graph on the $12n$-gon. If such a graph exists we have by Equation~\ref{3reg} that each loop face encloses $x=5n/12$ vertices and that the third face encloses $x=2n/12$ vertices. Even with this restriction we have been unable to construct a bifocal on the $12$-gon, and invite the reader to consider the following:

\begin{question}
Does the doubled regular $12$-gon admit a geodesic net modeled on a bifocal graph?
\end{question}







\section{Figure-Eight Graphs on Doubled Regular Polygons}

Next we consider the existence of figure-eight graphs on the doubled regular polygons. The figure-eight graphs have a single vertex of degree 4, hence the methods of the previous section (and in particular Equation~\ref{3reg}) do not apply.

First we construct a geodesic net modeled on a figure-eight graph for any doubled regular odd-gon. Depicted in Figure~\ref{odd}, this net is a self-intersecting closed geodesic which leaves the midpoint of an edge, traverses a far edge perpendicularly, returns to the original midpoint (but not smoothly), and then mirrors this behavior before closing smoothly. Note that this construction does not work on even-gons because geodesics crossing a side perpendicularly will close up after repeating this behavior on the opposite (and parallel) side.

\begin{figure}[ht!]
\centering
\includegraphics[scale=.54]{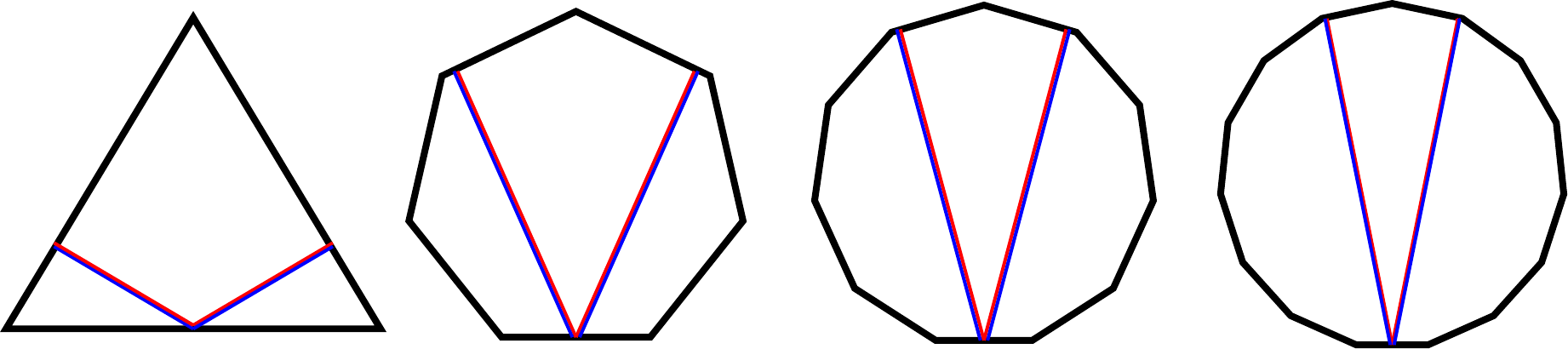}
\caption{Figure-eights on doubled regular odd-gons; image credit \cite{SUMRY18}.}
\label{odd}
\end{figure}

In order to address the existence of figure-eights on the doubled regular even-gons we again employ the version of Gauss-Bonnet from Equation~\ref{gb}. Considering one of the loop faces (where $y=1$), we produce the following restriction on the exterior turning angle: 
\begin{equation}\label{even}
\alpha = 2\pi - \frac{4\pi}{n}x.
\end{equation}

When $n=4$ this restriction (together with the fact that turning angles must satisfy $0<\alpha<\pi$) allows us to show the non-existence of figure-eight graphs: 

\begin{proposition}
The doubled regular 4-gon does not admit figure-eight graphs. 
\end{proposition}

\begin{proof}
When $n=4$ we have from Equation~\ref{even} that $\alpha = 2\pi - \pi x$ which means that the turning angle must be $0, \pi$, or $2\pi$. These are all degenerate turning angles and we conclude that 4-gons do not admit figure-eights. 
\end{proof}

When $n=6$ the same analysis as above yields that the turning angle in the loop face must be $\alpha=2\pi/3$. To our great surprise we were able to produce such a figure-eight graph on the doubled regular 6-gon.

\begin{proposition}
The doubled regular 6-gon admits a figure-eight graph. 
\end{proposition}

\begin{proof}
The proof is by construction:~see Figure~\ref{six6}. Symmetry makes it clear that this is indeed a geodesic net.
\end{proof}

\begin{figure}[ht!]
\centering
\includegraphics[scale=.3]{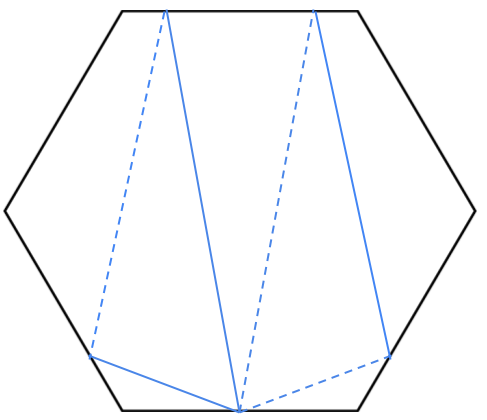}
\caption{Figure-eight on doubled regular 6-gon.}
\label{six6}
\end{figure}

Using Equation~\ref{even} in this way becomes less tenable when $n$ gets large, as the number of possible turning angles increases. Indeed, we again invite the reader to consider the following:

\begin{question}\label{q1}
Do the doubled regular even-gons with $n \geq 8$ admit geodesics nets based on a figure-eight graph?
\end{question}

Note that the technique of ``cutting off'' the corners used in Theorem 3.1 does not extend the doubled 6-gon result to doubled $6n$-gons since the figure-eight on the 6-gon intersects the polygon away from the midpoints of edges. We made some progress on this question, and even convinced ourselves of the non-existence of figure-eights on the doubled regular 8-gons, but were unable to generalize the result. We provide some notes here for the reader, along with a sketch of the 8-gon non-existence argument.

First note in regard to Question~\ref{q1} that it is sufficient to show the non-existence of figure-eights whose unique point of self-intersection lies at the midpoint of an edge. The single degree four vertex of a figure-eight must have opposite angles equal to satisfy the balancing condition of geodesic nets, and thus a figure-eight loop can be viewed as a single closed geodesic with exactly one point of self-intersection. We may restrict our attention to figure-eights whose unique point of self-intersection lies at the midpoint of an edge because any closed geodesic is in a family of homotopic closed geodesics whose union is a \textit{cylinder}, an open, connected strip bounded by parallel geodesic segments with endpoints at singularities (see explanation in Section $1$ of \cite{Vo}). Each cylinder on a doubled polygon contains the midpoint of an edge (see explanation in Section $2$ of \cite{ASZ}, who cite Veech \cite{V}), and thus each family of closed geodesics contains one that passes through the midpoint of an edge. Any figure-eight containing a midpoint of an edge must have its unique point of self-intersection at that midpoint by symmetry. As the number of self-intersections of a closed geodesic is preserved within its cylinder, any figure-eight lives in a cylinder that contains a figure-eight whose unique self-intersection lies at the midpoint of an edge. We conclude that it suffices to show the non-existence of such curves.

For the doubled regular 8-gons (where Gauss-Bonnet tells us that the only non-degenerate turning angle is $\alpha=\pi/2$) we were able to combine the above analysis with results from Section 4 of \cite{SUMRY18} to produce a finite list of possible figure-eight curves. We then used a computer program to sketch the development from each of these finite possibilities, and concluded that none were realizable. As this process did not feel entirely rigorous, we decided to include this non-existence result here in these notes, as opposed to a formal statement with a proof.

\section*{Acknowledgements and Declarations}
The authors would like to thank the Summer Undergraduate Mathematics Research at Yale (SUMRY) REU program that hosted and funded this research group during the summer of 2020. The authors have no additional funding to declare. All authors contributed equally to all parts of this manuscript. The authors have no conflicts of interest to declare. There is no data relevant to this project to share.


\begin{thebibliography}{9}

\bibitem{Ade}
I. Adelstein, {\it Minimizing geodesic nets and critical points of distance.} Diff. Geo. and Apps., 70: 101624, 2020. 

\bibitem{SUMRY18}
I. Adelstein, A. Azvolinsky, J. Hinman, and A. Schlesinger,
{\it Minimizing closed geodesics on polygons and disks.}
Involve: A Journal of Mathematics, 14: 11-52, 2021.

\bibitem{HM}
J. Hass and F. Morgan, {\it Geodesic nets on the 2-sphere.} Proc. AMS, 124(12): 3843–3850, 1996.

\bibitem{V}
W. A. Veech, {\it Teichmüller curves in moduli space, Eisenstein series and an application to triangular billiards.} Invent. Math., 97(3):553-583, 1989.

\bibitem{ASZ}
P. Apisa, R. M. Saavedra, and C. Zhang, {\it Periodic points on the regular and double $n$-gon surfaces.} Preprint. \url{https://arxiv.org/pdf/2011.02668.pdf}

\bibitem{Vo}
Y. Vorobets, {\it Periodic geodesics on translation surfaces.} Preprint. \url{https://arxiv.org/pdf/math/0307249.pdf}

\bibitem{NP}
A. Nabutovsky and F. Parsch, {\it Geodesic nets: some examples and open problems.} Experimental Mathematics, 2020.



\end{thebibliography}
\end{document}